\documentclass[11pt]{article}

\usepackage{amsmath,amsthm,amssymb}
\usepackage{mathrsfs,bm}
\usepackage{framed}
\usepackage{mathtools, enumitem}
\usepackage{cases}

\usepackage{float}

\usepackage{tikz-cd}
\usepackage{booktabs}

\usepackage{hyperref}
\hypersetup{%
colorlinks=true,
linkcolor=blue
}

\usepackage{tikz}
\usepackage[a4paper,top=2.5cm,bottom=2.5cm,left=3cm,right=3cm]{geometry}

\makeatletter
\newcommand{\leqnomode}{\tagsleft@true\let\veqno\@@leqno}
\makeatother

\usepackage{expl3}
\ExplSyntaxOn
\cs_new_eq:NN \Repeat \prg_replicate:nn
\ExplSyntaxOff

\newcommand{\jbk}[1]{\left\langle {#1} \right\rangle}

\newcommand{\rmop}[1]{\mathop{\mathrm{#1}}}

\newcommand{\beq}{\begin{equation}}
\newcommand{\eeq}{\end{equation}}

\newcommand{\nv}{\nu}

\newcommand{\tr}{\rmop{tr}}

\newcommand{\p}{\partial}

\newcommand{\Rbb}{\mathbb{R}}
\newcommand{\Sbb}{\mathbb{S}}

\newcommand{\Xfrak}{\mathfrak{X}}

\newcommand{\Ga}{\alpha}

\newcommand{\Gve}{\varepsilon}

\newcommand{\Gg}{\gamma}

\newcommand{\Gn}{\eta}

\newcommand{\Gv}{\nu}

\newcommand{\Gt}{\theta}

\newcommand{\Gs}{\sigma}
\newcommand{\Gsc}{\sigma_\mathrm{c}}
\newcommand{\Gsm}{\sigma_\mathrm{m}}
\newcommand{\Gss}{\sigma_\mathrm{s}}

\newcommand{\Gj}{\tau}

\newcommand{\Go}{\omega}

\newcommand{\GD}{\Delta}

\newcommand{\GG}{\Gamma}

\newcommand{\GO}{\Omega}

\def\XXint#1#2#3{{\setbox0=\hbox{$#1{#2#3}{\int}$}
    \vcenter{\hbox{$#2#3$}}\kern-.5\wd0}}

\newtheorem{prop}{Proposition}[section]
\newtheorem{theo}[prop]{Theorem}

\newtheorem{lemm}[prop]{Lemma}
\newtheorem{rema}[prop]{Remark}

\theoremstyle{definition}

\newtheorem*{note*}{Note}
\newtheorem*{claim*}{Claim}

\newtheorem*{rema*}{Remark}
\newtheorem*{exer*}{Exercise}
\newtheorem*{prob*}{Problem}

\numberwithin{equation}{section}

\begin{document}

\title{Ellipsoidal characterization of neutral inclusions for imperfect bonding of high-conductivity type\thanks{Y. Ji was supported by KIAS Individual Grant no. MG089002 at Korea Institute for Advanced Study. S. Sakaguchi was partially supported by JSPS KAKENHI Grant Number JP26K06856. X. Li was supported by Zhejiang Provincial Natural Science Foundation of China under Grant No. LMS25A010005.}}

\author{
	Junyong Eom\thanks{
		Division of Pure and Applied Science, Graduate School of Science and Technology,
		Gunma University, Gunma 376-8515, Japan.
		Emails: \texttt{eom@gunma-u.ac.jp} (J. Eom) and
		\texttt{shota.fukushima@gunma-u.ac.jp} (S. Fukushima).
	}
	\and
	Shota Fukushima\footnotemark[2]
	\and
	Yong-Gwan Ji\thanks{
		School of Mathematics, Korea Institute for Advanced Study,
		Seoul 02455, S. Korea.
		Email: \texttt{ygji@kias.re.kr}.
	}
	\and
	Hyeonbae Kang\thanks{
		Department of Mathematics and Institute of Applied Mathematics,
		Inha University, Incheon 22212, S. Korea.
		Email: \texttt{hbkang@inha.ac.kr}.
	}
	\and
	Xiaofei Li\thanks{
		School of Mathematical Sciences, Zhejiang University of Technology,
		Hangzhou, 310023, P. R. China.
		Email: \texttt{xiaofeilee@hotmail.com}.
	}
	\and
	Shigeru Sakaguchi\thanks{
		Admissions Center, Tohoku University, Sendai 980-8576, Japan.
		Email: \texttt{sigersak@tohoku.ac.jp}.
	}
}

\maketitle

\begin{abstract}
This paper concerns neutral inclusions for imperfect bonding of high-conductivity type. An inclusion, which is a bounded domain, is said to be of imperfect bonding of high-conductivity type if the flux is discontinuous along its boundary while the potential is continuous. The inclusion is neutral to a uniform field if the presence of the inclusion does not perturb the field outside the inclusion. It is known that ellipses and ellipsoids can be neutral to all uniform fields by introducing a proper imperfect bonding coefficient on boundaries. The purpose of this paper is to prove the converse. We prove that if an inclusion of imperfect bonding of high-conductivity type is neutral to all uniform fields, then it is an ellipse or an ellipsoid. The neutrality condition is given by existence of the solution to a certain differential equation on the boundary surface and the main result is proved by converting the neutrality condition into an algebraic boundary identity characterizing ellipses and ellipsoids. 
\end{abstract}



\tableofcontents

\section{Introduction and statements of results} \label{sec:intro}

This paper concerns inclusions with imperfect bonding interface properties of high-conductivity type which is characterized by continuity of the potential and discontinuity of the flux along the boundary. We prove that if such an inclusion is neutral to multiple uniform fields, then it is of elliptic or ellipsoidal shape.

To describe results in a precise manner, let $D \subset \Rbb^d \ (d=2,3)$ be a bounded domain with Lipschitz boundary, which will be referred to as an inclusion. Let
\begin{align}\label{def_sigma}
\Gs(x) = \begin{cases}
\Gsc \quad &x \in D, \\
\Gsm \quad &x \in \Rbb^d \setminus \overline{D},
\end{cases}
\end{align}
where $\Gsc$ and $\Gsm$ are symmetric positive-definite constant matrices such that $\Gsc \neq \Gsm$. This $\Gs$ depicts a two-phase material with constant conductivity matrices in the inclusion and the background. We consider the differential equation
\beq\label{PDE}
\operatorname{div} (\sigma \nabla u)=0 \ \text{ in } \Rbb^d\setminus \partial D
\eeq
with the decay condition at $\infty$:
\beq\label{decay}
u(x)-\jbk{c, x}= O(|x|^{1-d}) \ \text{ as } |x|\to \infty,
\eeq
where $c \in \Rbb^d$ is a constant vector. Two kinds of interface conditions can be assigned on $\p D$: the perfect and imperfect bonding conditions. We say $\p D$ satisfies the perfect bonding condition if the potential and the flux are continuous along $\p D$, namely,
\beq\label{PB}\tag{PB}
u|_+=u|_-,  \ \jbk{\Gsc \nabla u|_-, \nu} = \jbk{\Gsm \nabla u|_+, \nu} \ \text{ on } \partial D.
\eeq
Here and throughout this paper the subscript $-(+)$ stands for the limit to $\p D$ from the inside(outside) $D$ and $\nu$ for the outward unit normal to $\p D$. The imperfect bonding condition is characterised by either discontinuity of the potential or the flux along $\p D$. If the potential is discontinuous along $\p D$, it is called the Low-Conductivity (LC) type and described as
\beq\label{LC}\tag{LC}
u|_+ - u|_- =  \Gg\jbk{\Gsm \nabla u|_+, \Gv},  \ \jbk{\Gsc \nabla u|_-, \nu} = \jbk{\Gsm \nabla u|_+, \nu} \ \text{ on } \partial D
\eeq
for some function $\Gg$ on $\p D$. If the flux is discontinuous along $\p D$, it is called the High-Conductivity (HC) type and described as
\beq\label{HC}\tag{HC}
u|_+ = u|_- ,  \ \jbk{\Gsc \nabla u|_-, \nu} - \jbk{\Gsm \nabla u|_+, \nu}= \operatorname{div}_{S} (\alpha \nabla_{S} u) \ \text{ on } \p D,
\eeq
for some function $\Ga$ on $\p D$, where $\operatorname{div}_{S}$ and $\nabla_{S}$ stand for the surface divergence and gradient on $\p D$, respectively, where definitions are ginven in Section \ref{sect:surface:geometry}.

We now turn to describe the notion of neutrality. For a fixed nonzero vector
\(c\in\mathbb R^d\), the triple $(D, \Gs, \Ga)$ is said to be HC-type neutral to the uniform field $-c$ if there exists a solution $u$ to \eqref{PDE} with imperfect bonding interface condition \eqref{HC} such that
\beq\label{neutral}
u(x) \equiv\jbk{c, x}, \ \ x \in \Rbb^d \setminus \overline{D},
\eeq
instead of \eqref{decay}. This means that the presence of inclusion does not perturb the uniform field outside the inclusion. If this holds for every \(c\in\mathbb R^d\), then \((D,\sigma,\Ga)\) is said to be HC-type neutral to \textit{multiple} uniform fields. Similarly, the triple \((D,\sigma,\Gg)\) is said to be LC-type neutral
to a uniform field, or to multiple uniform fields, if the same condition
holds with \eqref{HC} replaced by \eqref{LC}.

The imperfect bonding conditions of LC- and HC-type were derived in \cite{BM1999} as limits of the core-shell structures when the thickness of the shell tends to $0$ in two different ways: the conductivity of the shell tends to zero or to infinity. The former regime gives rise to \eqref{LC}, whereas the latter to \eqref{HC}. It is shown in the same paper that, under suitable conditions, ellipsoids can be neutral to multiple uniform fields for both LC- and HC-type interfaces. We present the explicit relations of the ellipsoids with the conductivities and the function $\Ga$ in Section \ref{sec:explicit}. We also provide an elementary proof of neutrality in the same section. We mention in particular that if \((D,\sigma,\Ga)\) ($D$ is an ellipsoid) is HC-type neutral to multiple uniform fields, then $\Gsm-\Gsc$ is either positive- or negative-definite, and $\Ga>0$ if $\Gsm-\Gsc>0$, $\Ga<0$ if $\Gsm-\Gsc<0$. The corresponding result for balls was proved earlier in \cite{T.R}.

There are LC- and HC-type neutral inclusions of non-ellipsoidal shape to a single uniform field; see \cite{BM1999}. It is helpful to mention that the two-phase material with the perfect bonding condition is never neutral. In fact, the solution $u$ admits the dipole expansion
$$
u(x)-\jbk{c, x} = \frac{1}{\Go_d} \frac{\jbk{x, Mc}}{|x|^d} + O(|x|^{-d}) \ \mbox{ as } |x| \to \infty,
$$
where $M$ is a matrix and $\Go_d=2\pi$ if $d=2$ and $4\pi$ if $d=3$. The matrix $M$, which is called the polarization tensor, is positive- or negative-definite depending on positivity or negativity of $\Gsm-\Gsc$ (see, for example, \cite{AK07}).

In this paper, we address the converse question: is there a shape other than an ellipsoid which is LC- or HC-type neutral to multiple uniform fields? The answer is no. We prove the following theorem.

\begin{theo}\label{characterization_HC}
Let $D\subset \Rbb^d$ ($d=2,3$) be a bounded domain with connected $C^2$ boundary $\partial D$. In the case \(d=3\), we further assume that $\p D$ is of genus zero. If there exists $\Ga\in C^1(\partial D)$ with $\Ga >0$ (resp. $\Ga<0$) such that $(D, \sigma, \Ga)$ is HC-type neutral to multiple uniform fields, then $A$ defined by
	\begin{align}\label{def_A}
	A := \frac{\tr \left(\Gsm - \Gsc\right)}{d-1}\operatorname{Id}-\left(\Gsm - \Gsc\right)
	\end{align}
is positive-definite (resp. negative-definite), and
$\p D$ is an ellipsoid (or an ellipse) of the form
	\begin{equation}\label{eq:HC:ellipsoid}
	\p D = \left\{x \in \Rbb^d \;|\; \jbk{x, Ax}-2\jbk{b, x}=c_0 \right\},
	\end{equation}
	for some constant vector $b\in \Rbb^d$ and constant $c_0\in \Rbb$.

\end{theo}

Suppose that $(D, \Gs, \Ga)$ is HC-type neutral to multiple uniform fields. Then the solution to \eqref{PDE} with \eqref{decay} satisfies \eqref{neutral}.
By the first condition in \eqref{HC}, $u(x)= \jbk{c,x}$ in $\Rbb^d$. By the second condition in \eqref{HC}, we have
\begin{equation}\label{1000}
\jbk{(\Gsc -\Gsm) c, \nu} = \operatorname{div}_S (\alpha \nabla_S \jbk{c,x})
\end{equation}
for all $c \in \Rbb^d$.
We prove that this relation implies
\beq\label{3000}
Ax-\alpha\nu = \text{const.} \quad\text{on } \p D
\eeq
with $A$ defined in \eqref{def_A}. This property guarantees that $A$ is positive- or negative-definite depending on the sign of $\Ga$ and $\p D$ is an ellipsoid (see Lemma \ref{lemm:chracterization:ellipsoid}).

It is easy to see that LC-type neutral inclusions are ellipsoids.
In fact, if $(D, \sigma, \Gg)$ is LC-type neutral to multiple uniform fields, one can see from \eqref{LC} that
$$
(\Gsm^{-1}-\Gsc^{-1}) x - \Gg \nu= \text{const.} \quad\text{on } \p D,
$$
and hence $\p D$ is an ellipsoid. This was also proved in \cite[Appendix]{KL2019} using a different characterization of ellipsoids.

There is yet another way to achieve neutral inclusions; by means of core-shell structures (three-phase material). Let $D$ and $\GO$ be bounded domains in $\Rbb^d$ such that $\overline{D} \subset \GO$. Then $D$ is the core and $\GO \setminus \overline{D}$ is the shell. The conductivities of the core and shell are assumed to be scalars denoted by $\Gsc$ and $\Gss$; that of the matrix is a symmetric positive-definite matrix denoted by $\Gsm$. Such core-shell structures can be neutral to multiple uniform fields: it is proved for concentric balls with isotropic $\Gsm$ in \cite{H1} and for confocal ellipsoids with anisotropic $\Gsm$ in \cite{kerker} (see also \cite{JKLS} for explicit relations among the confocal ellipsoids and the conductivities that ensure neutrality). The question of the converse has been looked into. It is proved in \cite{KL2d} (in \cite{MS} for the case when $\Gsc=0$ or $\infty$) that in two-dimensions the only core-shell structure neutral to multiple uniform fields is of the form of confocal ellipses. The three-dimensional case remains unsolved. The only case solved is when $\Gsm$ is isotropic and $\Gsc > \Gss$. In such a case, it is proved in \cite{KLS3d} that core-shell structure is of the form of concentric balls. We refer to a review paper \cite{JKLS} for a related over-determined problem for confocal ellipsoids.

It is worthwhile to mentioning recent developments of gradient estimates when the interface condition is given by either the perfect bonding or the imperfect bonding of LC-type with a constant bonding coefficient $\Gg$. If inclusions whose conductivity is $\infty$ (perfect conductor) with the perfect bonding interface condition are closely located, then the gradient of the solution becomes arbitrarily large as the distance $\Gve$ between inclusions tends to zero. There is a huge literature on this subject for which we conveniently refer to recent surveys \cite{DY, Kang} and references therein. We also refer to \cite{DLY24, DLY25} for recent development on the insulation case. Recently the following question is investigated: what happens if the perfect bonding condition is replaced with LC-type imperfect bonding condition. In the recent paper \cite{FJKL}, the case when inclusions are two-dimensional discs of the same radius has been considered. It is proved that if the imperfect bonding coefficient $\Gg$ is positive, then the gradient remains bounded uniformly in $\Gve$. This result has been generalized in \cite{DYZ26} to include higher dimensions and more general shapes. In \cite{DLZ}, precise gradient estimates are obtained in terms of $\Gg$ and $\epsilon$. It is proved in \cite{FJK} that, as $\gamma$ tends to zero, the solution corresponding to the imperfect bonding coefficient $\gamma>0$ converges in various spaces to the solution corresponding to $\gamma=0$, namely, the solution of the perfect bonding problem.

This paper is organized as follows. In Section 2, we collect basic notions and facts of surface geometry to be used in later sections. In Section 3, we review the fact from \cite{BM1999} that ellipsoids can be HC-type neutral to multiple fields. We include a proof which is simpler than the proof in \cite{BM1999}. Section 4 is devoted to the proof of Theorem \ref{characterization_HC}.

\section{Basic facts on surface geometry}\label{sect:surface:geometry}
In this section, we collect basic notions and facts from surface geometry. See, for example, \cite[Chapter 8]{Lee2018}. For simplicity of presentation, let $\Gamma \subset \mathbb{R}^d$ be an oriented closed $C^\infty$-surface. All the formulas established here remain valid for surfaces with lower regularity, subject to the minimum regularity assumptions at each step. 

We denote by \(\mathfrak{X}(\GG)\) the space of $C^{\infty}$ tangent vector fields on \(\Gamma\), namely,
\begin{align}\label{def:frakX}
\mathfrak{X}(\GG)
:=
\left\{
X \in C^{\infty}(\GG, \Rbb^d) \;:\; X(x)\in T_x\GG
\text{ for every } x\in\Gamma
\right\},
\end{align}
where $T_x \Gamma$ is the tangent space at $x \in \GG$.

For an $\Rbb^m$-valued function $F\in C^1 (\Gamma, \Rbb^m)$ ($m \geq 1$) and $v\in T_x \Gamma$, we define the differential of $F$ along $v$ by
\begin{equation}\label{eq:diff:defi}
D_v F (x): =\left.\frac{d}{dt} F(r(t)) \right|_{t=0}\in \Rbb^m,
\end{equation}
where \(r\) is any smooth curve on \(\Gamma\) satisfying \(r(0)=x\) and \(r'(0)=v\). Note that
this definition is independent of the choice of the curve $r$.  For $X \in \mathfrak{X}(\GG)$, we write
\[
(D_XF)(x):=\left(D_{X(x)}F\right)(x),\qquad x\in\Gamma.
\]

The first and second fundamental forms of \(\Gamma\) are defined by
\[
\mathrm{I}(X,Y):=\langle X,Y\rangle,
\qquad
\mathrm{II}(X,Y):=\langle D_XY,\nu\rangle,
\qquad X,Y\in\mathfrak X(\Gamma),
\]
where the equalities are understood pointwise on \(\Gamma\).

Let
\beq\label{defP}
P=P_x:=\operatorname{Id}-\nu_x\nu_x^T: \Rbb^d \longrightarrow T_x \Gamma
\eeq
be the orthogonal projection onto \(T_x\Gamma\). For $X, Y\in \Xfrak (\Gamma)$, we define
\beq\label{covform}
\nabla^\Gamma_X Y:=P(D_X Y).
\eeq
This $\nabla^\Gamma$ satisfies the following fundamental properties, which characterize $\nabla^\Gamma$ as the Levi--Civita connection associated with the metric $\operatorname{I}$.

\begin{itemize}
	\item (Compatibility) For $X$, $Y$, $Z\in \Xfrak (\Gamma)$,
	\begin{equation} \label{eq:metric_compatibility}
	D_X (\operatorname{I} (Y,Z)) =
	\operatorname{I} (\nabla^\Gamma_XY, Z)+ \operatorname{I} (Y,\nabla^\Gamma_X Z).
	\end{equation}
	
	\item (Torsion-free property) For $X, Y\in \Xfrak (\Gamma)$,
	\begin{equation}
	\label{eq:torsion:free}
	\nabla^\Gamma_X Y -\nabla^\Gamma_Y X = [X, Y],
	\end{equation}
	where \([\cdot,\cdot]\) denotes the Lie bracket of vector fields on
	\(\Gamma\).
\end{itemize}

For a \((0,2)\)-tensor $B: \mathfrak{X} (\GG) \times \mathfrak{X} (\GG) \to C^{\infty}(\GG)$, we denote by $T_B: \mathfrak{X} (\GG) \to \mathfrak{X} (\GG)$ the
corresponding \((1,1)\)-tensor obtained by raising an index with respect to
\(\mathrm{I}\), that is,
\[
\mathrm{I}(T_B X,Y)=B(X,Y),
\qquad X,Y\in\mathfrak X(\Gamma).
\]
The \((1,1)\)-tensor corresponding to the second fundamental form \(\mathrm{II}\) is called the shape operator, also called the Weingarten map, and we denote it by \(W\):
\begin{align}\label{def:shapeop}
W:= T_{\mathrm{II}}.
\end{align}

For a \((0,2)\)-tensor $B$, we define its trace at each $x\in\Gamma$
to be the trace of the corresponding $(1,1)$-tensor $T_B$, \textit{i.e.}, 
\begin{align}
\left(\operatorname{tr}_S B\right)(x) := \left(\operatorname{tr}_S T_B\right)(x). 
\end{align}
The subscript $S$ is used to emphasize that the trace is taken over the tangent space $T_x\Gamma$, rather than over the ambient space $\mathbb{R}^d$. If we take an orthonormal basis $\{ e_j\}_{j=1}^{d-1}$ of $T_x \Gamma$, then this trace is computed by $\sum_{j=1}^{d-1} \mathrm{I}(T_B e_j,e_j)$. Now let $A:\mathbb{R}^d\to\mathbb{R}^d$ be a constant matrix. For each $x \in T_x \Gamma$, $PA$ is a linear map on $T_x \Gamma$. Since $Pe_j = e_j$, its trace is given by
\begin{equation}\label{eq:trace:Gamma}
\begin{aligned}
\operatorname{tr}_S PA = \sum_{j=1}^{d-1}\jbk{PAe_j,e_j} = \sum_{j=1}^{d-1}\jbk{Ae_j,e_j} = \tr A -  \jbk{\Gv, A \Gv},
\end{aligned}
\end{equation}
where $\tr$ denotes the usual trace of $A$ as a linear map on $\mathbb{R}^d$. 

The shape operator $W$ is symmetric with respect to the metric $\operatorname{I}$, \textit{i.e.},
\begin{align}
\mathrm{I}(WX,Y) = \mathrm{I}(X,WY).
\end{align}
This fact implies that all eigenvalues of $W$ are real. The arithmetic average of the eigenvalues of $W$, counted with multiplicities, is the mean curvature, namely,
\begin{equation}\label{def:mean_curvature}
H(x):=\frac{\left(\operatorname{tr}_S W\right)(x)}{d-1}.
\end{equation}
With this convention, a sphere oriented by its outward unit normal has negative mean curvature.
The shape operator also appears as the differential of the unit normal vector field:
\begin{equation}\label{eq:Gauss:der}
WX = - D_X \nu.
\end{equation}
Therefore, we have
\beq\label{mean}
\nabla \cdot \Gv = -\left( d-1 \right)H.
\eeq

We now cite a result on a Codazzi tensor that plays a crucial role in the proof. For a \((0,2)\)-tensor $B$, the covariant derivative is defined by
		\begin{align}\label{def:corvariant_deri:B}
		\left(\nabla_X^{\GG} B\right)(Y,Z):= D_X(B(Y,Z)) - B(\nabla_{X}^{\GG} Y, Z) - B(Y, \nabla_{X}^{\GG} Z).
		\end{align}
$B$ is said to be a Codazzi tensor if
		\begin{align}
		\left(\nabla_X^{\GG} B\right)(Y,Z) = \left(\nabla_Y^{\GG} B\right)(X,Z).
		\end{align}

The following theorem can be found in \cite{LSW} (see also \cite{Fox2013}).

\begin{theo}\label{theo:trace-free-Codazzi-sphere} If $\GG$ is $C^\infty$-diffeomorphic to the unit sphere $\Sbb^2$ and $B$ is symmetric, trace-free, and Codazzi $(0,2)$-tensor, then $B = 0$.
\end{theo}	

For $f\in C^2 (\Gamma, \Rbb)$ and $F\in C^1 (\Gamma, \Rbb^d)$, which is not necessarily tangential to $\partial D$, surface gradient $\nabla_S$, surface divergence $\operatorname{div}_S$, and surface Laplacian (also called Laplace--Beltrami operator) $\GD_S$ are defined by
\begin{align}
\label{eq:surface:grad:Cartesian}  \nabla_S f &= \nabla f - \jbk{\nu, \nabla f} \nu, \\
\label{eq:surface:div:Cartesian}  \operatorname{div}_S F &= \operatorname{div} F - \jbk{\nu, (\nabla F)\nu}, \\
\GD_S f \label{eq:surface:Laplacian:Cartesian} &= \operatorname{div}_S \nabla_S f = \GD f - \jbk{\Gv, (\nabla^2 f) \Gv} - \left( \nabla \cdot \Gv\right) \p_\Gv f ,
\end{align}
where $f$ is extended as a $C^2$-function and $F$ and $\nu$ are extended as $C^1$-functions in a neighborhood
of $\Gamma$, and $\nabla^2 f$ is the Hessian matrix of $f$ (see, for example, \cite{Reilly}). These expressions are independent of the chosen extensions. If $F$ is tangential to $\partial D$, then $\operatorname{div}_S F$ corresponds to the intrinsic surface divergence, which is defined purely in terms of the first fundamental form on $\Gamma$. The following product rule for the surface differential operator holds:
\begin{align}\label{eq:prod:surf:div}
 \operatorname{div}_S \left(\Ga \nabla_S f\right) = \jbk{\nabla_S \Ga, \nabla_S f}  + \Ga \GD_S f.
\end{align}

\section{Neutrality relation for ellipsoids}\label{sec:explicit}

After a rotation, if necessary, we may assume that $\Gsm-\Gsc$ is diagonal, namely,
$$
\Gsm - \Gsc = \operatorname{diag}(q_1 , \ldots, q_d) \quad  (d=2,3).
$$
Let $\p E$ be the ellipsoid defined by
\beq \label{def_ellipsoid}
	\p E = \left\{ x \in \Rbb^d \; | \;\sum_{j=1}^{d} \frac{x_j^2}{r_j^2} = 1 \right\} \quad  (d=2,3). 	
\eeq
The relation of $\Gsm - \Gsc$ and the ellipsoid $\p E$ for HC-type neutrality is given by
\beq\label{HCrel}
q_1 \left( \sum_{i \neq 1} \frac{1}{r_i^2} \right)^{-1} = \ldots = q_d \left( \sum_{i \neq d} \frac{1}{r_i^2} \right)^{-1}.
\eeq
If this relation holds, the function $\Ga$ is defined by
\beq\label{alpha_ellipsoid}
\Ga(x) = q_k \left( \sum_{i \neq k} \frac{1}{r_i^2} \right)^{-1} w(x),
\eeq
where
\beq\label{def_w}
	w(x) = \sqrt{\sum_{j=1}^{d}\frac{x_j^2}{r_j^4}}.
\eeq

\begin{theo}\label{HCell}
$(E, \Gs, \Ga)$ is HC-type neutral to multiple uniform fields.
\end{theo}

This theorem is proved in \cite{BM1999}. We include a simpler proof in this section. A few remarks are in order before giving the proof.

\begin{rema}
\begin{enumerate}[label=\textnormal{(\roman*)}]
	\item The condition \eqref{HCrel} shows that $\Gsm-\Gsc$ is either positive- or negative-definite. If it is positive-definite, then $\Ga>0$; if it is negative-definite, then $\Ga<0$.
	
	\item $\Gsm-\Gsc$ is isotropic, namely, $q_1=\ldots=q_d$ if and only if $r_1^2=\ldots=r_d^2$, namely, $E$ is a disk or a ball.
	
	\item In two dimensions, the equation \eqref{HCrel} can be solved for $r_j^2$ as long as $q_1, \ q_2$ have the same sign. It means that for given $\Gs$ we can find an ellipse which become neutral with respect to $\Gs$ (and vice versa). However, this is not true in three-dimensions. In addition to definiteness, $\Gsm-\Gsc$ needs to satisfy the following conditions:
	\beq\label{tri}
	|q_1|<|q_2|+|q_3|, \quad |q_2|<|q_1|+|q_3|, \quad |q_3|<|q_1|+|q_2|.
	\eeq
	In fact, \eqref{HCrel} yields that
	$$
	|q_k| = t \sum_{i \neq k} \frac{1}{r_i^2}
	$$
	for some $t>0$. Thus we have
	$$
	\sum_{i=1}^3 |q_i| = 2t \sum_{i=1}^3 \frac{1}{r_i^2}
	$$
	and \eqref{tri} follows.
	
\end{enumerate}

\end{rema}

\noindent \textit{Proof of Theorem \ref{HCell}.}
For $k \in \{ 1,\ldots,d \}$, let $u_k(x) = x_k$ ($x \in \Rbb^d$). It is enough to show that $u_k$ satisfies the imperfect bonding condition \eqref{HC}, namely,
$$
\jbk{\left. \Gsc \nabla u_k \right|_{-}, \Gv} - \jbk{\left. \Gsm \nabla u_k \right|_{+}, \Gv} = \operatorname{div}_S \left(\Ga \nabla_S u_k \right),
$$
which amounts to proving
\beq
\operatorname{div}_S \left(\Ga \nabla_S u_k \right)= - q_k \nu_k \quad\text{on } \p E.
\eeq

According to \eqref{eq:prod:surf:div}, it holds that
$$
\operatorname{div}_S \left(\Ga \nabla_S u_k \right) = \jbk{\nabla_S \Ga, \nabla_S u_k}  + \Ga \GD_S u_k.
$$
By \eqref{eq:surface:grad:Cartesian} and \eqref{eq:surface:Laplacian:Cartesian}, we have
$$
\nabla_S \Ga = \nabla \Ga - (\p_\nu \Ga) \nu, \quad  \nabla_S u_k = e_k - \nu_k \nu, \quad \GD_S u_k = - \left( \nabla \cdot \Gv\right) \nu_k
$$
where $e_k$ is the unit vector whose $k$-th component is $1$. Thus we have
\beq\label{5000}
\operatorname{div}_S \left(\Ga \nabla_S u_k \right) = \frac{\p \Ga}{\p x_k} - (\partial_\nu\alpha)\nu_k + \alpha(-\nabla\cdot \nu)\nu_k.
\eeq

Let
$$
\eta_k:= q_k \left( \sum_{i \neq k} \frac{1}{r_i^2} \right)^{-1}
$$
so that $\Ga(x) = \eta_k w(x)$ ($w$ is defined in \eqref{def_w}).

The outward unit normal $\Gv$ to $\p E$ is given by
$$
\Gv(x) = \left(\frac{x_1}{r_1^2 w}, \cdots,\frac{x_d}{r_d^2 w}\right)
$$
and
\beq\label{5100}
\frac{\partial w}{\partial x_k} = \frac{x_k}{r_k^4 w}= \frac{\nu_k}{r_k^2}.
\eeq
Thus, we have
\beq\label{5200}
\p_\nu \Ga (x) = \frac{\Gn_k}{w^2}\sum_{i=1}^d \frac{x_i^2}{r_i^6}.
\eeq
We also have
$$
	\frac{\partial \nu_k}{\partial x_k}
	= \frac{1}{r_k^2 w}
	+ \frac{x_k}{r_k^2}\left(-\frac{1}{w^2}\right)\frac{x_k}{r_k^4 w},
$$
and hence
\beq\label{5300}
\nabla\cdot \nu
	= \frac{1}{w}\sum_{i=1}^d \frac{1}{r_i^2}
	- \frac{1}{w^3}\sum_{i=1}^d \frac{x_i^2}{r_i^6}.
\eeq
Plugging identities \eqref{5100}, \eqref{5200} and \eqref{5300} into \eqref{5000}, we obtain 	
	\begin{align*}
	& \operatorname{div}_S \left(\Ga \nabla_S u_k \right) \\
	&= \Gn_k\frac{\nu_k}{r_k^2}
	- \left(\frac{\Gn_k}{w^2}\sum_{i=1}^d \frac{x_i^2}{r_i^6}\right)\nu_k + \Gn_k w\left(
	-\frac{1}{w}\sum_{i=1}^d \frac{1}{r_i^2}
	+ \frac{1}{w^3}\sum_{i=1}^d \frac{x_i^2}{r_i^6}
	\right)\nu_k
	\\
	&= -\Gn_k \left(\sum_{i \neq k} \frac{1}{r_i^2}\right) \nu_k \\
	&=-q_k \Gv_k.
	\end{align*}
This completes the proof.
\qed

\section{Proof of Theorem \ref{characterization_HC}}

We begin with the following elementary geometric lemma.
\begin{lemm}\label{lemm:chracterization:ellipsoid}
Let $D\subset \Rbb^d$ ($d\geq 2$) be a bounded domain with $C^1$-smooth connected boundary $\partial D$. If
	\begin{align}\label{eq:geometric:bdry}
	Ax - b = \Ga\Gv \quad \text{ on } \p D,
	\end{align}
	where $A$ is a symmetric matrix, $b$ is a constant vector, and $\Ga$ is a positive (resp. negative) function, then $A$ is a positive-definite (resp. negative-definite) matrix and $\p D$ is an ellipsoid of the form \eqref{eq:HC:ellipsoid} for some constant $c_0$.
\end{lemm}
\begin{proof}
	Define a function $f:\Rbb^d \to \Rbb$ by
	\begin{align}
	f(x):=\jbk{x, Ax}-2\jbk{b, x}.
	\end{align}
	Taking a tangential derivative along an arbitrary tangential vector $v$ of $\partial D$ at $x$, we obtain from \eqref{eq:geometric:bdry} that
	\begin{align}
	D_v f(x)=2\jbk{v, Ax}-2\jbk{v, b}=2\Ga(x)\jbk{v, \nv_x}=0.
	\end{align}
	Since $\p D$ is connected, $f(x) = c_0$ on $\partial D$ for some constant $c_0$, \textit{i.e.}, $\p D$ is a quadric. Since $D$ is bounded, the quadric $\partial D$ must be an ellipsoid and $A$ is positive or negative definite. If $g$ is positive, the identity $\nabla f(x)=2 \Ga(x)\nv_x$ ($x\in \partial D$) shows that the direction in which $f$ increases must point toward the exterior of ellipsoid $\partial D$. Thus $A$ must be positive definite and $\p D$ is the ellipsoid of the form \eqref{eq:HC:ellipsoid}. If $\Ga$ is negative, the same argument applied to \(-f\) shows that \(-A\) is positive
	definite, and hence \(A\) is negative definite and $\p D$ is the ellipsoid of the form \eqref{eq:HC:ellipsoid}.
\end{proof}

Suppose that $(D, \Gs, \Ga)$ is HC-type neutral to multiple uniform fields and $\Ga >0$ or $\Ga < 0$.
Then \eqref{1000} holds.
After a rotation, if necessary, we may assume that $\Gsm - \Gsc$ is a diagonal matrix, and write
\beq\label{eq:def_Q}
Q:= \Gsm - \Gsc = \operatorname{diag}(q_1, \ldots, q_d).
\eeq
Then, \eqref{1000} can be rephrased as
\beq\label{2000}
-q_k \nu_k = \operatorname{div}_S (\alpha \nabla_S u_k), \quad k=1, \ldots, d,
\eeq
where $u_k(x) = x_k$ ($x \in \Rbb^d$).
Let $A$ be the matrix defined in \eqref{def_A}.
We prove \eqref{3000} from which Theorem \ref{characterization_HC} follows (by Lemma \ref{lemm:chracterization:ellipsoid}).

We begin with the two-dimensional case, which is straightforward.

\medskip

\noindent\textit{Proof for the two-dimensional case}. If $d=2$, then
$A=\operatorname{diag}(q_2, q_1)$.
Let $\tau=(-\Gv_2, \Gv_1)^T$. Then the tangential projection $P$ is given by
\begin{align}
P = \Gj \Gj^T.
\end{align}
Let $x=x(s)$ be the arclength parametrization on $\p D$, oriented so that
\[
\frac{dx}{ds}=\tau.
\]
Then one can see from \eqref{eq:surface:grad:Cartesian}--\eqref{eq:surface:Laplacian:Cartesian} that
\begin{align*}
\nabla_S f = \left(\frac{d f}{ds} \right) \Gj, \ \
\operatorname{div}_S F = \left(\frac{d F}{ds}\right) \cdot \Gj, \ \
\label{surface_div_ftau} \operatorname{div}_S \left(f\Gj\right) = \frac{df}{ds}, \ \
\GD_S f &= \frac{d^2f}{ds^2} .
\end{align*}

Since $\nabla_S u_1 = -\Gv_2 \Gj$ and $\nabla_S u_2 = \Gv_1 \Gj$, it follows from \eqref{2000} that
$$ 
-q_1 \Gv_1 = \frac{d}{ds} (-\Ga \Gv_2), \quad
-q_2 \Gv_2 = \frac{d}{ds} (\Ga \Gv_1).
$$
We then infer from these identities that
$$
\frac{d}{ds}(Ax) = A\tau = A\begin{pmatrix}
-\Gv_2 \\
\Gv_1
\end{pmatrix} = \frac{d}{ds} \left( \Ga \Gv \right).
$$
Since $\p D$ is connected, we conclude that $Ax -  \Ga\Gv$ is constant on $\p D$. This completes the proof for the two-dimensional case.

\medskip

\noindent\textit{Proof for the three-dimensional case}. If $d=3$, then $A=\operatorname{diag}(a_1,a_2,a_3)$, where
$$
a_1=\frac{-q_1 +q_2+q_3}{2},
\ \
a_2=\frac{q_1-q_2+q_3}{2},
\ \
a_3=\frac{q_1+q_2-q_3}{2}.
$$

Let $e_k$ be the unit vector whose $k$-th component is $1$ as before. Since $\nabla_S u_k= e_k - \Gv_k \Gv$, we have from \eqref{mean}, \eqref{eq:prod:surf:div} and \eqref{2000} that
$$
-q_k \Gv_k =  \left(\nabla_S \Ga\right)_k + 2 \Ga H \Gv_k, \quad k=1, \ldots, d,
$$
or equivalently, in matrix form,
\begin{align}
-Q\Gv = \nabla_S \Ga + 2\Ga H \Gv \quad\text{on } \p D.
\end{align}
It then follows that
\begin{align}
\label{eq:neutrality:normal} -\jbk{\Gv, Q\Gv} &= 2 \Ga H   \\
\label{eq:neutrality:tangential} -PQ\Gv &= \nabla_S \Ga.
 \end{align}

We first prove that $\p D$ and $\Ga$ are of class $C^{\infty}$. In fact, if $\p D$ is of class $C^{k,\Gt}$ for some $k\geq 2$ and $0\leq \Gt<1$ (here $C^{k, 0}=C^k$), then $\Gv$ is of class $C^{k-1,\Gt}$, and we have from definition \eqref{defP} of $P$ that $PQ\Gv$ is of class $C^{k-1,\Gt}$. We then infer from \eqref{eq:neutrality:tangential} that $\Ga$ is of class $C^{k,\Gt}$. Thus we infer from \eqref{eq:neutrality:normal} that $H$ is $C^{k-1,\Gt}$. Since the mean curvature equation is quasilinear and uniformly elliptic for each classical solution, this implies by means of interior Schauder estimates that $\p D$ is $C^{k+1,\Gt}$ when $0<\theta<1$, and $C^{k, \tau}$ for any $0< \tau < 1$ when $\theta=0$ (see, for example, \cite[Theorem 6.17]{GT}). Iterating the above argument shows that both $\partial D$ and $\alpha$ are of class $C^\infty$.

It follows from \eqref{eq:trace:Gamma} and \eqref{eq:neutrality:normal} that
\begin{equation} \label{eq:trace:Gamma:PM}
\begin{aligned}
\operatorname{tr}_S PA &= \tr A - \jbk{\Gv, A\Gv} \\
&= a_1+a_2+a_3 - a_1 \Gv_1^2 - a_2 \Gv_2^2 - a_3 \Gv_3^2 \\
&=(a_2+a_3)\Gv_1^2 + (a_1+a_3)\Gv_2^2 + (a_1+a_2)\Gv_3^2 \\
&= \jbk{\Gv, Q\Gv} \\
&= -2 \Ga H
\end{aligned}
\end{equation}
and from \eqref{eq:neutrality:tangential} that
\begin{align} \label{eq:PMnu}
PA\Gv = A\Gv - \jbk{\Gv, A\Gv} \nu= -PQ\Gv = \nabla_S \Ga.
\end{align}

Set $\GG = \p D$. Define a function $F: \GG \to \Rbb^3$ by
\begin{equation}\label{eq:def_F}
F(x):=Ax-\alpha\nu.
\end{equation}
Then, we have from the product rule and the equation \eqref{eq:Gauss:der} that
\begin{align}
D_X F = AX - \jbk{X,\nabla_S\Ga}\Gv + \Ga WX
\end{align}
Therefore, we have from \eqref{eq:PMnu} and the facts $P^T = P$, $PX = X$, and $A^T = A$ that
\begin{equation}\label{eq:representation:F}
\begin{aligned}
D_X F &= AX - \jbk{X,PA\Gv}\Gv + \Ga WX \\
&= AX - \jbk{AX, \Gv}\Gv + \Ga WX \\
&= PAX + \Ga WX.
\end{aligned}
\end{equation}

Since $\GG$ and $\alpha$ are of class $C^\infty$, the projection $P$
and the shape operator $W$ are smooth on $\Gamma$. Thus $D_X F \in C^\infty(\Gamma)$. Moreover, since $W$ maps from $\mathfrak X(\Gamma)$ to $\mathfrak X(\Gamma)$, we have $D_X F = PAX+\alpha WX\in\mathfrak X(\Gamma)$.

We define a $(0,2)$-tensor $B: \mathfrak{X} (\GG) \times \mathfrak{X} (\GG) \to C^{\infty}(\GG)$ and a corresponding $(1,1)$-tensor $T=T_B: \mathfrak{X} (\GG) \to \mathfrak{X} (\GG)$  by
\begin{align}
\label{def:B} B(X,Y) &:= \langle TX, Y\rangle \\
TX &:= D_X F= PAX + \Ga WX.
\end{align}

Suppose that $B$ is a symmetric, trace-free, Codazzi tensor. Since $\GG$ is a connected $C^\infty$ surface of genus zero, the classification theorem for compact surfaces implies that $\GG$ is $C^\infty$-diffeomorphic to the unit sphere $\mathbb{S}^2$ (see, for example, \cite{WB}). Therefore, we infer from Theorem \ref{theo:trace-free-Codazzi-sphere} that $B=0$, or equivalently $D_X F=0$. Since $X$ is arbitrary and $\GG$ is conneted, we conclude that $F$ is constant on $\GG$ and the proof is completed by Lemma \ref{lemm:chracterization:ellipsoid}. So it sufficient to prove that $B$ is a symmetric, trace-free, Codazzi tensor. 

Since $P^T = P$, $PY= Y$, \(A^T=A\) and \(W\) is symmetric, we have
\begin{equation}
\begin{aligned}
B(X,Y)
&= \langle PAX + \Ga WX,Y\rangle \\
&=\langle AX,Y\rangle+\alpha\langle WX,Y\rangle \\
&=\langle X,AY\rangle+\alpha\langle X,WY\rangle \\
&=\langle AY,X\rangle+\alpha\langle WY,X\rangle \\
&=B(Y,X).
\end{aligned}
\end{equation}
Thus $B$ is symmetric.

Let \(x\in\Gamma\), and let
	\(\{e_1,e_2\}\) be an orthonormal basis of \(T_x\Gamma\). Then, we have from \eqref{def:mean_curvature} and \eqref{eq:trace:Gamma:PM} that
	\begin{equation}
	\begin{aligned}
	\operatorname{tr}_S B
	=\operatorname{tr}_S T
	=
	\underbrace{\operatorname{tr}_S PA}_{=-2\alpha H}
	+
	\alpha\underbrace{\operatorname{tr}_S W}_{=2H} =0.
	\end{aligned}
	\end{equation}
Thus $B$ is trace-free.

Since $TY \in \mathfrak{X}(\GG)$, we have from the metric compatibility \eqref{eq:metric_compatibility} that
	\begin{equation*}
	\begin{aligned}
	D_X(B(Y,Z)) &= D_X(I(TY,Z)) \\
	&=\jbk{\nabla_{X}^{\GG}\left(TY \right), Z} + \jbk{TY, \nabla_{X}^{\GG} Z}.
	\end{aligned}
	\end{equation*}	
	Therefore, we have from the definitions of the covariant derivative of $B$ in \eqref{def:corvariant_deri:B} and $B$ in \eqref{def:B} that
	\begin{align}\label{eq:nabla:XB}
	\left(\nabla_X^{\GG} B\right)(Y,Z) &=\jbk{\nabla_{X}^{\GG}\left(TY \right), Z} - \jbk{T\left(\nabla_{X}^{\GG}Y\right), Z }.
	\end{align}

We now use the standard commutation identity for directional derivatives:
\begin{align*}
D_XD_YF-D_YD_XF-D_{[X,Y]}F=0.
\end{align*}
Since $D_Z F = TZ \in \mathfrak{X}(\GG)$ for $Z \in \mathfrak{X}(\GG)$,  taking tangential projection $P$ in above equation, we have from the definition of $\nabla^{\GG}$ in \eqref{covform} that
	\begin{equation}\label{eq:Codazzi:T}
	\begin{aligned}
	0 &=PD_X (TY) - PD_Y (TX) - PT[X,Y] \\
	&= \nabla_X^{\GG} (TY) - \nabla_Y^{\GG} (TX) - T[X,Y].
	\end{aligned}
	\end{equation}
	Therefore, we have from the equations \eqref{eq:nabla:XB} and \eqref{eq:Codazzi:T}, and torsion-free property \eqref{eq:torsion:free} that
	\begin{align*}
	&\left(\nabla_X^{\GG} B\right)(Y,Z) - \left(\nabla_Y^{\GG} B\right)(X,Z) \\
	&= \langle \underbrace{\nabla_{X}^{\GG}\left(TY \right) - \nabla_{Y}^{\GG}\left(TX \right)}_{=T[X,Y]} , Z \rangle - \langle T\underbrace{\left(\nabla_{X}^{\GG}Y - \nabla_{Y}^{\GG}X\right)}_{=[X,Y]}, Z \rangle = 0
	\end{align*}
Thus $B$ is a Codazzi tensor.

\end{document}